\documentclass[a4paper,11pt,reqno,oneside]{amsart}
\pdfoutput=1
\usepackage[utf8]{inputenc}
\usepackage[english]{babel}
\usepackage[T1]{fontenc}
\usepackage{amsaddr}
\usepackage[top=3.5cm, bottom=3.5cm, left=3cm, right=3cm]{geometry}
\usepackage{color}    % farver!
\usepackage{graphicx} % pics
\usepackage{tikz}
\usetikzlibrary{math}
\usepackage{amsmath,amssymb,mleftright} % amsmath&symboler
\usepackage{amsthm}
\usepackage{listings} % kildekode visning
\usepackage{parcolumns} % minipage
\usepackage{array,tabu,booktabs} % flotte tabeller
\usepackage{multirow} % multirow i tabeller
\usepackage{lmodern}  % Latin Modern font
\usepackage{enumitem} % for list: \begin{enumerate}[label=\Alph*]
\usepackage{hyperref} % url
\usepackage{breakurl}
\usepackage{diagbox}
\usepackage{hhline}
\usepackage{bigints}
\usepackage[makeroom]{cancel}
\usepackage{setspace}
\usepackage{mathtools}
\mathtoolsset{showonlyrefs}
\newcommand{\lp}{\left(}
\newcommand{\rp}{\right)}

\newcommand{\R}{\mathbb{R}}

\newcommand{\pfcrho}{\square}
\newcommand{\pfa}{\blacksquare}
\newtheorem{theorem}{Theorem}[section]
\newtheorem{conjecture}{Conjecture}[section]
\newtheorem{corollary}[theorem]{Corollary}

\newtheorem{lemma}[theorem]{Lemma}
\newtheorem{proposition}[theorem]{Proposition}
\newtheorem{remark}[theorem]{Remark}

\allowdisplaybreaks

\newcommand{\Lop}[3]{\ensuremath{#1\Delta_{#2}^{#3}\,}}

\usepackage{microtype}

\usepackage[backend=biber,
            style=alphabetic,
            abbreviate=false,
            dateabbrev=false,
            alldates=long]{biblatex}
\usepackage{csquotes}

\author{Joonas Ilmavirta$^1$, Hj\o rdis Schl\"uter$^{1,*}$}
\address{$^1$Department of Mathematics and Statistics, University of Jyv\"askyl\"a, Finland}
\address{$^*$Corresponding author}
\email{joonas.ilmavirta@jyu.fi, hjordis.a.schluter@jyu.fi}

\subjclass{35R30}
\keywords{Dirichlet-to-Neumann map, inverse problems, elastic wave equation, Riemannian geometry, anisotropic stiffness tensor.}

\title[Gauge freedoms in the anisotropic elastic DN map]{Gauge freedoms in the anisotropic elastic Dirichlet-to-Neumann map}
\date{\today}

\begin{document}

\maketitle

\begin{abstract}
    We address the inverse problem of recovering the stiffness tensor and density of mass from the Dirichlet-to-Neumann map. We study the invariance of the Euclidean and Riemannian elastic wave equation under coordinate transformations. Furthermore, we present gauge freedoms between the parameters that leave the elastic wave equations invariant. We use these results to present gauge freedoms in the Dirichlet-to-Neumann map associated to the Riemannian elastic wave equation. 
\end{abstract}

%\tableofcontents
%\newpage

%\NTR{We have indicated all changes through these footnotes.}

\section{Introduction}
We study the elastic wave equation (EWE) in the $n$-dimensional Euclidean space and on an $n$-dimensional Riemannian manifold. We address the inverse problem of recovering the anisotropic stiffness tensor and the density of mass from the Dirichlet-to-Neumann (DN) map. In this context it is essential for the reconstruction procedure whether the DN map determines the stiffness tensor and density uniquely. For that purpose we study invariance of the elastic wave equation: How the EWE behaves under coordinate transformations and what gauge freedoms there are between the parameters. Based on these results we present gauge freedoms for the DN map in the Riemannian setting and conjecture that these are the correct and full gauge groups. If this holds true then the Euclidean DN map determines the stiffness tensor and density uniquely in two (and possibly higher) dimensions.  This is of particular interest as the three-dimensional Euclidean EWE is the natural setting in seismology, where the EWE models seismic waves. \par

Let $M\subset\R^n$ be the closure of a smooth domain, and let~$g$ be a Riemannian metric on~$M$.
On~$M$ we use the Euclidean coordinates~$x$.
The material parameters are the stiffness tensor~$c$ (contravariant of rank~$4$) and the density~$\rho$.
Following the symmetries of the stress and strain tensors and their relationship, the stiffness tensor has the minor and major symmetries 
    \begin{equation}\label{csym}
        c^{ijk\ell}=c^{jik\ell}=c^{k\ell ij},
    \end{equation}
and is positive in the sense that\footnote{Here and throughout the paper we use the Einstein summation convention, also in the Euclidean setting where all indices are kept down. Whenever there is a non-repeated index, the corresponding equation holds for all values $1,\dots,n$ of it.}
    \begin{equation}\label{cpos}
        c^{ijk\ell}A_{ij}A_{k\ell} \geq \delta g^{jk}g^{i\ell} A_{ij}A_{k\ell}
    \end{equation}
for some $\delta>0$ and for all symmetric matrices~$A$. Additionally, the density of mass $\rho$ is positive:
\begin{equation*}
    \rho \geq K,
\end{equation*}
for some constant $K>0$.

See~\cite{dehoop2023reconstruction} for a discussion on these structures.
The density is assumed positive in the usual sense.
The metric tensor is not a property of the material but of the space itself, and it is therefore forced to be Euclidean in practice --- this will play an important role.

These material parameters give rise to the elastic Laplacian~$\Lop{}{c}{g}$ defined by
\begin{align}
\label{eq:EWop}
(\Lop{}{c}{g} u(t,x))^i
=
\frac{1}{\sqrt{\det(g)}}  \, 
\partial_{x^j}
\lp
\sqrt{\det(g)}
c^{ijk\ell} g_{\ell m} \partial_{x^k} u^m(t,x) 
\rp,
\end{align}
which maps vector fields to vector fields.
The elastic wave operator is
\begin{equation}
P_{(\rho,c,g)}
=
\partial_t^2 - \Lop{\rho^{-1}}{c}{g}.
\end{equation}
The displacement vector $u(t,x)\in\R^n$ satisfies the Riemannian version $P_{(\rho,c,g)}u=0$ of the elastic wave equation.

We are therefore naturally led to the boundary value problem
\begin{equation}
\label{REWE}
     \begin{cases}
         P_{(\rho,c,g)} u(t,x) = 0 &  \text{in }   (0,T) \times M,\\
         u(t,x) = f(t,x) & \text{on } (0,T) \times \partial M ,\\
         u(0,x)=\partial_t u(0,x)=0  & \text{in }M
     \end{cases}
\end{equation}
in the spacetime $(0,T) \times M$ and where $f$ satisfies: $f(x,0)=\partial_t f(x,0)=0$ for $x\in \partial M$.

For the special case $g=g_E$ the Riemannian EWE reduces to the familiar Euclidean EWE
\begin{equation}\label{EWE}
        \begin{cases}
            (P_{(\rho,c,g_E)}u(t,x)) = 0 &  \text{in } (0,T) \times M\\
            u(t,x) = f(t,x) & \text{on } (0,T) \times \partial M ,\\
         u(0,x)=\partial_t u(0,x)=0  & \text{in }M.
        \end{cases}
\end{equation}
The Euclidean elastic Laplacian acts as
\begin{equation*}
    (\Lop{}{c}{g_E} u(t,x))_i = \partial_{x_j}
    \lp
    c_{ijk\ell} \, \partial_{x_k} u_{\ell}(t,x)
    \rp.
\end{equation*}
For the Euclidean elastic wave operator~$P_{(\rho,c,g_E)}$ we will use the short hand notation $P_{(\rho,c,g_E)}=P_{(\rho,c)}$ in the following. Assuming that $c_{ijk\ell} \in L^{\infty}(M), \rho \in L^{\infty}(M)$ and with $f \in L^2(0,T;H^1(\partial M)) \cap H^1(0,T;L^2(\partial M))$ the initial value boundary value problem \eqref{EWE} has the unique solution $u \in L^2(0,T;H^1(M))\cap H^1(0,T;L^2(M))$ and additionally, $u \in C(0,T;H^1(M))\cap C^1(0,T;L^2(M))$ \cite[Thm. 2.4.5]{stolk00},\cite[Lemma 3.5]{BeLa02}. Following the analysis by Stolk in \cite[Sec. 2.3]{stolk00} these well-posedness results can be extended to the Riemannian initial boundary value problem \eqref{REWE}.

\subsection{The Finsler metric arising from elasticity}

Inverse problems related to elasticity are mainly concerned with recovering the stiffness tensor~$c$ or the reduced stiffness tensor $a=\rho^{-1}c$ everywhere from boundary data. One way to define elastic geometry is in terms of travel time distance between two points. Here one considers an elastic body that is modeled as a manifold and distance is measured by the shortest time it takes for an elastic wave to go from one point to the other. When the elastic material is elliptically anisotropic the elastic geometry is Riemannian, but this puts very stringent restrictions on the stiffness tensor.

Recent research is devoted to the fully anisotropic setting for the stiffness tensor, where only the physically necessary assumptions as in \eqref{csym}--\eqref{cpos} are needed.
The resulting elastic geometry is not Riemannian but Finslerian \cite{HILSmay21,HILSaug21}.
The Finsler metric can be derived from the principal behavior of the elastic wave operator $P_{(\rho,c,g)}$ defined above.
The components of its matrix-valued principal symbol $\sigma(t,x,\omega,p)$ are
\begin{equation*}
    \sigma(P_{(\rho,c,g)})_{\,\,m}^i(t,x,\omega,\xi)
    =
    -\delta_{\,\,m}^i \omega^2  + \rho^{-1}(x)c^{ijk\ell}(x) g_{\ell m}(x) \xi_j \xi_k,
\end{equation*}

with $(t,x,\omega,\xi)\in T^*((0,T) \times M)$.
Introducing the slowness vector $p=\omega^{-1}\xi$ and the Christoffel matrix~$\Gamma(x,p)$ defined by
\begin{equation*}
     \Gamma_{\,\,m}^i(x,p):=\rho^{-1}(x)c^{ijk\ell}(x)g_{\ell m}(x) p_j p_k,
\end{equation*}
one can rewrite the principal symbol as
\begin{equation*}
    \sigma(P_{(\rho,c,g)})(t,x,\omega,p):=\omega^2 [\Gamma(x,p)-I].
\end{equation*}
By symmetry and positivity of~$c$ in \eqref{csym}--\eqref{cpos} the Christoffel matrix~$\Gamma$ is symmetric  and positive definite.
Eigenvectors of~$\Gamma$ are called polarizations and the corresponding eigenvalues are related to the wave speeds.
%This implies that all eigenvalues of~$\Gamma$ are positive and they are related to different polarizations.
It turns out that the Finsler metric~$F_a^{qP}$ associated to the $qP$-polarization (the fastest waves) and the density-normalized stiffness tensor field $a=\rho^{-1}c$ for the elastic geometry can be derived from the largest eigenvalue of~$\Gamma$ as described in~\cite{HILSaug21}. 

%Further research has shown in the two-dimensional Euclidean setting that in order for two elastic wave operators $P_{\rho_1,c_1}$ and $P_{\rho_2,c_2}$ to give rise to the same principal symbol the diffeomorphism $\phi$ that maps $(\rho_1,c_1)$ to $(\rho_2,c_2)$ must be conformal \cite{JonnePaper}. This is conjectured to also hold true in higher dimensions. Motivated by this observation we investigate what conformal freedoms there are between the parameters $\rho, c$ and $g$ in order to give rise to the same principal behavior and thus the same Finsler metric. Furthermore, we investigate the conformal freedoms on the full level of the PDE.

\subsection{Dirichlet-to-Neumann maps and inverse problems}
\label{sec:dn-intro}

The (hyperbolic) Dirichlet-to-Neumann (DN) map is the map that sends the Dirichlet boundary condition~$f$ to its corresponding Neumann boundary condition. For the elastic wave operator~$P_{(\rho,c,g)}$ defined in~\eqref{REWE} the associated DN operator~$\Lambda_{(\rho,c,g)}$ reads
\begin{equation}\label{DNmapG}
        (\Lambda_{(\rho, c, g)} f)^i= \, \nu_j \,c^{ijk\ell} g_{\ell m} \partial_{x^{k}} u^{m}\Bigg\vert_{(0,T) \times \partial M}, 
\end{equation}
where~$\nu$ is the unit outward normal to~$\partial M$.
The inverse problem is to recover~$\rho$, $c$ and~$g$ from~$\Lambda_{(\rho,c,g)}$ (or only~$\rho$ and~$c$ when~$g$ is assumed known).

In the Euclidean setting the DN operator~$\Lambda_{(\rho, c)}$ associated to~$P_{(\rho,c)}$ defined in~\eqref{EWE} simplifies to
\begin{equation}\label{DNmapE}
        (\Lambda_{(\rho, c)} f)_i=\nu_j \, c_{ijk\ell} \, \partial_{x_{k}} u_{\ell}  \Bigg\vert_{(0,T) \times \partial M}.
\end{equation}
The inverse problem is to recover~$\rho$ and~$c$ from~$\Lambda_{(\rho,c)}$.
Following \cite{HoopNakamuraZhai19} the operator $\Lambda_{(\rho, c)}$ is defined as the map $\Lambda_{(\rho, c)}: C^2 \lp 0,T;H^{\frac{1}{2}}(\partial M) \rp \rightarrow L^2 \lp 0,T;H^{-\frac{1}{2}}(\partial M) \rp$. In the same lines as assessing well-posedness for the Riemmannian initial boundary value problem, the results for the Euclidean DN operator $\Lambda_{(\rho, c)}$ can be extended to the Riemannian DN operator $\Lambda_{(\rho, c, g)}$.
%{\color{red} Following the well-posedness of the initial value problem \eqref{EWE}, the linear operator $\Lambda_{(\rho, c)}$ is defined as the map $\Lambda_{(\rho, c)}: L^2(0,T;H^1(\partial M)) \cap H^1(0,T;L^2(\partial M)) \rightarrow L^2(0,T;L^2(\partial M))$ \cite[Lemma 2.3]{GuoZhang09},\cite[Thm 1]{BeLa02}. Using the energy estimate by Stolk \cite[eq. (2.22)]{stolk00} one can prove that $\Lambda_{(\rho, c)}$ is bounded. In the same lines as assessing well-posedness for the Riemmannian initial boundary value problem, the results for the Euclidean DN operator $\Lambda_{(\rho, c)}$ can be extended to the Riemannian DN operator $\Lambda_{(\rho, c, g)}$.}

This inverse problem has been addressed for the case when the stiffness tensor is isotropic~\cite{Rachele001,RACHELE002,HoopNakamuraZhai17,Stefanov2017} and for the case when the stiffness tensor is transversely isotropic or orthorhombic (with additional assumptions on the symmetry axis and symmetry planes) ~\cite{HoopNakamuraZhai19,deHoop2019}.
Very recent results address the reconstruction procedure for anisotropic stiffness tensors~\cite{dehoop2023reconstruction}.

For the DN maps corresponding to the Calderón problems for conductivities and Riemannian and Lorentzian metrics, see Appendix~\ref{app:dn}.

\subsection{Acknowledgements}

This work was supported by the Research Council of Finland (Flagship of Advanced Mathematics for Sensing Imaging and Modelling grant 359208 and Centre of Excellence of Inverse Modelling and Imaging 353092 and other grants 351665, 351656, 358047).
We thank the referee for useful suggestions and feedback.

%The paper is organized as follows: In Section 2, we present our results split into two parts. In part one we present our main result concerning gauge freedoms in the Dirichlet-to-Neumann map. In part two we address invariance of the elastic wave equation: We state the explicit pushforwards of the stiffness tensor under a coordinate transformation and the conformal freedoms between the parameters that leave the equation invariant. In section 3 we use the results from the second part of section 2 to prove our main result.

%In the following we present the recent results concerning principal behavior of the two-dimensional Euclidean setting. Furthermore, we present conjectures concerning the higher dimensional Euclidean and the general Riemannian setting. These conjectures are followed by the results concerning coordinate invariance of the Euclidean and Riemannian EWE. We conclude listing the conformal freedoms between the parameters for the different settings and relate these to the explicit definition of the pushforward of the stiffness tensor in the Euclidean setting. 

\section{Results}

All material parameters, metric tensors, and diffeomorphisms are assumed to be~$C^\infty$ throughout the paper.
The only possibly non-smooth function is the displacement field~$u$.

\subsection{Dirichlet-to-Neumann maps}
\label{sec:ThmConj}

The point of this paper is to identify the correct gauge freedom in the inverse boundary value problems posed above.
The natural gauges appear to be surprisingly different in Euclidean and Riemannian geometries, differing not only by changes of coordinates.
%We do not prove that these are the correct and full gauge groups, but we feel confident enough to guess.

See section~\ref{sec:pf-def} for the definitions of the various pushforwards.

In the following we assume that the qP geometry is simple.
\begin{theorem}\label{thm:RiemConf}
Let $(M,g)$ be a Riemannian manifold of any dimension $n\geq2$ with boundary, let~$c$ be a stiffness tensor satisfying the symmetry and positivity conditions in \eqref{csym}--\eqref{cpos}, and let $\rho>0$ be the density of mass. Let $\phi\colon  M\to  M$ be a diffeomorphism fixing the boundary and let $\mu \neq 0$ be a function so that $\mu \vert_{\partial M}=1$. Then the DN map defined in~\eqref{DNmapG} satisfies
\begin{equation}\label{eq:gaugefreedoms}
    \Lambda_{\left(\mu^{\frac{n}{2+n}} \,\phi_* \rho, \mu \, \phi_* c, \mu^{-\frac{2}{2+n}} \, \phi_* g \right)}=\Lambda_{(\rho, c, g)}.
\end{equation}
\end{theorem}

The previous observation concerns gauge freedoms in the DN map and on the full level of the PDE. The following recent observation concerns gauge freedoms on the principal level of the PDE:
\begin{remark}[{\cite{IlmavirtaCocan}}]\label{rem:Jonne}
Let $\Omega \subset \R^d$ with $d=2$ or $d=3$ be a bounded, smooth, connected, and simply connected domain. There is an open and dense set~$W$ of density-normalized stiffness tensor fields on~$\Omega$ so that if $a\in W$ and $\phi^* F_a^{qP}=F_b^{qP}$ with a diffeomorphism $\phi\colon\bar\Omega\to\bar\Omega$ fixing the boundary, then $\phi=\operatorname{id}$ and $b=a$.
\end{remark}
The core of the proof of Remark~\ref{rem:Jonne} is to show that~$\phi$ has to be conformal.
Then it follows from fixing the boundary that it has to be the identity.

Note that the Finsler metric $F_a^{qP}$ is linked to the Christoffel matrix and thus the principal symbol of the elastic wave operator $P_{(\rho,c)}$. Therefore this result suggests that the principal behaviour of $P_{(\rho,c)}$ determines $a=\rho^{-1}c$; cf. part~(\ref{part:c}) of Proposition~\ref{prop:CoordInvar}. Information about the density is contained in lower order terms; cf. Lemma~\ref{lem:fullPDE}.

Using the fact that only conformal diffeomorphisms preserve the principal behavior of the elastic wave operator, we can derive the Euclidean version of Theorem~\ref{thm:RiemConf}: 
In the Euclidean setting $g_1=g_2=g_E$, so the metric component of the conclusion of Conjecture~\ref{conj:R} becomes $g_E=\mu^{-2/(2+n)}\phi_*g_E$.
This makes~$\phi$ a conformal map of $(\bar\Omega,g_E)$ fixing the boundary, which in fact forces~$\phi$ to be the identity map\footnote{Suppose there is a conformal map $\phi\colon\bar\Omega\to\bar\Omega$ fixing the boundary. By the Riemann mapping theorem there is a conformal map $\eta\colon D\to\Omega$ from the unit disc, and by smoothness of~$\Omega$ it extends to a map $\bar\eta\colon\bar D\to\bar\Omega$. Thus $\bar\eta^{-1}\circ\phi\circ\bar\eta\colon\bar D\to\bar D$ is a conformal map of the disc fixing its boundary. But the conformal self-maps of the disc are the M\"obius transformations, and it is easily checked that only the trivial one fixes the boundary. Thus $\bar\eta^{-1}\circ\phi\circ\bar\eta=\operatorname{id}_D$ and $\phi=\operatorname{id}_\Omega$.} and $\mu=1$. This yields:

\begin{corollary}\label{cor:EucConf}
Let $\Omega\subset\R^d$ with $d=2$ or $d=3$ be a bounded, smooth, connected, and simply connected domain, let~$c$ be a stiffness tensor field satisfying the symmetry and positivity conditions in \eqref{csym}--\eqref{cpos}, and let $\rho>0$ be the density of mass. Then the gauge freedoms for the DN map $\Lambda_{\left(\rho,c \right)}$ by a diffeomorphism $\phi$ or a function $\mu$ in \eqref{eq:gaugefreedoms} are eliminated.
\end{corollary}
As Remark~\ref{rem:Jonne} implies that the principal behavior of $P_{(\rho,c)}$ determines $a=\rho^{-1}c$ uniquely, Corollary~\ref{cor:EucConf} implies that the most likely gauge freedom for the full behavior of $P_{(\rho,c)}$ is eliminated and we usually expect more gauge freedom on the principal level than the full level (cf. Table~\ref{tab:ConformalFreedoms}), we find it likely that the following result holds true:
\begin{conjecture}\label{conj:E}
Let $\Omega\subset\R^d$ with $d=2$ or $d=3$ be a bounded, smooth, connected, and simply connected domain.
Let $\rho_i>0$ be density fields and~$c_i$ be generic stiffness tensor fields that satisfy the symmetry and positivity conditions in \eqref{csym}--\eqref{cpos} for $i=1,2$.
If the DN maps defined in~\eqref{DNmapE} satisfy $\Lambda_{(\rho_1,c_1)}=\Lambda_{(\rho_2,c_2)}$, then $\rho_1=\rho_2$ and $c_1=c_2$.
\end{conjecture}

As a consequence we also find it likely that the gauge freedom found in Theorem~\ref{thm:RiemConf} is indeed the whole gauge in the Riemannian setting:

\begin{conjecture}\label{conj:R}
In the setting of Theorem~\ref{thm:RiemConf} with generic stiffness tensor fields~$c_i$, density fields~$\rho_i$ and Riemannian metrics~$g_i$ the following holds:
If $\Lambda_{(\rho_1,c_1,g_1)}=\Lambda_{(\rho_2,c_2,g_2)}$, then $(\rho_2,c_2,g_2)=\left(\mu^{\frac{n}{2+n}} \,\phi_* \rho_1, \mu \, \phi_* c_1, \mu^{-\frac{2}{2+n}} \, \phi_* g_1\right)$ for some function~$\mu$ and diffeomorphism~$\phi$.
\end{conjecture}

Note how the distinction between principal and full behavior and recovery of the normalized stiffness tensor $a=\rho^{-1}c$ and the full stiffness tensor should be compared with Rachele's work~\cite{Rachele001,RACHELE002,Rachele03} in the isotropic\footnote{The stiffness tensor of an isotropic material enjoys more symmetry than we assumed, and it can be parametrized by the two Lamé parameters~$\lambda$ and~$\mu$ as $c_{ijk\ell}=\lambda \delta_{ij}\delta_{k\ell}+\mu (\delta_{ik}\delta_{j\ell}+\delta_{i\ell}\delta_{jk})$. The wave speeds for pressure and shear waves are $c_P=\sqrt{(\lambda+2\mu)/\rho}$ and $c_S=\sqrt{\mu/\rho}$.} setting:

\begin{enumerate}
    \item The DN map determines the travel times between boundary points for both wave speeds~$c_P$ and~$c_S$.
    \item If the manifolds $(\bar\Omega,c_{P/S}^{-2}g_E)$ are simple, then this geometric information determines the Riemannian metrics uniquely.
    \item The metrics are known to be conformally Euclidean, which eliminates the coordinate gauge freedom intrinsic to a Riemannian manifold.
    Thus the DN map determines~$c_P$ and~$c_S$.
    \item The density~$\rho$ is a lower order term and can be recovered from the DN map via a ray transform.
\end{enumerate}

%\begin{theorem}[{\cite{...}}]\label{thm:Jonne}
%Let $\Omega \subset \R^2$ be a nice domain. There is an open and dense set $W$ of stiffness tensor fields so that if $a \in W$ and $F_a^{qP}=\phi^* F_b^{qP}$, then $a \equiv b$ and $\phi$ is conformal.
%\end{theorem}

\subsection{Various pushforwards}
\label{sec:pf-def}

We define various different kinds of pushforwards, some of which are somewhat non-standard.
Let $U,V\subset\R^n$ be any open sets and let $\phi\colon U\to V$ be a diffeomorphism.

For the stiffness tensor $c^{ijk\ell}$, the metric tensor $g_{ij}$ and the density $\rho$ we define the pushforwards $\phi_*c$, $\phi_*g$, and $\phi_*\rho$ in the usual fashion:
\begin{align}
\label{eq:pf-c/rhocg}
(\phi_* c)^{\hat{i}\hat{j}\hat{k}\hat{\ell}}
&=
\lp
c^{ijk\ell}
(D\phi)_{\hspace{2mm} i}^{\hat{i}}
(D\phi)_{\hspace{2mm} j}^{\hat{j}}
(D\phi)_{\hspace{2mm} k}^{\hat{k}}
(D\phi)_{\hspace{2mm} \ell}^{\hat{\ell}} 
\rp
\circ \phi^{-1}
,
\\
\label{eq:pf-g/rhocg}
(\phi_* g)_{\hat{i}\hat{j}}
&=
\lp
g_{ij}
(D\phi^{-1})_{\hspace{2mm} \hat{i}}^{i}
(D\phi^{-1})_{\hspace{2mm} \hat{j}}^{j}
\rp
\circ \phi^{-1}
,
\\
\label{eq:pf-rho/rhocg}
\phi_* \rho
&=
\rho
\circ \phi^{-1}
.
\end{align}
All the familiar formulas from differential geometry hold for these standard pushforwards --- or, equivalently, pullbacks over~$\phi^{-1}$.

When the map~$\phi$ is conformal with respect to the Euclidean metric, we define
\begin{align}
\label{eq:pf-c/rhoc}
(\phi_\pfcrho c)^{\hat{i}\hat{j}\hat{k}\hat{\ell}}
&=
\lp
\det(D\phi)^{-1-2/n}
c^{ijk\ell}
(D\phi)_{\hspace{2mm} i}^{\hat{i}}
(D\phi)_{\hspace{2mm} j}^{\hat{j}}
(D\phi)_{\hspace{2mm} k}^{\hat{k}}
(D\phi)_{\hspace{2mm} \ell}^{\hat{\ell}} 
\rp
\circ \phi^{-1}
,
\\
\label{eq:pf-rho/rhoc}
\phi_\pfcrho \rho
&=
\lp
\det(D\phi)^{-1}
\rho
\rp
\circ \phi^{-1}
\end{align}
and
\begin{align}
\label{eq:pf-c/c}
(\phi_\pfa c)^{\hat{i}\hat{j}\hat{k}\hat{\ell}}
&=
\lp
\det(D\phi)^{-2/n}
c^{ijk\ell}
(D\phi)_{\hspace{2mm} i}^{\hat{i}}
(D\phi)_{\hspace{2mm} j}^{\hat{j}}
(D\phi)_{\hspace{2mm} k}^{\hat{k}}
(D\phi)_{\hspace{2mm} \ell}^{\hat{\ell}} 
\rp
\circ \phi^{-1}
.
\end{align}
It is important to note that the usual pushforward~$\phi_*$ is defined on the triplet $(\rho,c,g)$, the pushforward~$\phi_\pfcrho$ on the pair $(\rho,c)$ and pushforward~$\phi_\pfa$ only on~$c$.

\begin{remark}
If we want to define a pushforward over a possibly non-conformal diffeomorphism~$\phi$, then the definitions of~$\phi_\pfcrho (c,\rho)$ and~$\phi_\pfa c$ need to be adjusted.
The correct formulas are obtained by treating the stiffness tensor as the mixed rank object $c^{ijk}_{\phantom{ijk}\ell}$ and applying the usual formulas from differential geometry.
Then the EWE is indeed invariant, but the issue is that the so obtained~$\phi_\pfcrho c$ no longer satisfies the symmetry conditions. It is shown in the proof of Remark~\ref{rem:Jonne} in \cite{IlmavirtaCocan} that only conformal distortions can give rise to a valid slowness polynomial that corresponds to a stiffness tensor with the required symmetries.
For a conformal map~$\phi$ this pushforward does preserve symmetry, and it gives rise exactly to our equation~\eqref{eq:pf-c/rhoc}.

One way to see this is as follows:
The usual Euclidean EWE is not coordinate invariant because it is tied to the Euclidean concept of distance.
Conformal coordinate invariance only follows because Lemma~\ref{lem:fullPDE} allows shifting scalar factors within the triplet $(\rho,c,g)$.
\end{remark}

\begin{remark}
\label{rmk:cov-div}
The elastic wave operator~\eqref{eq:EWop} on a Riemannian manifold can be written as

\begin{equation}\label{eq:EWop2}
(\Lop{}{c}{g} u(t,x))^i
=
\nabla_{x^j}
\lp
c^{ijk\ell} g_{\ell m}
\nabla_{x^k} u^m(t,x) 
\rp
\end{equation}
using covariant derivatives.
Formula \eqref{eq:EWop} and \eqref{eq:EWop2} are equivalent; the only difference is in the coordinate representation of the covariant divergence. Note that these two ways to define the elastic wave operator correspond to the following equivalent definition for the Laplace-Beltrami operator: $\Delta_g f=\frac{1}{\sqrt{\det g}}\partial_{x^i}\left(\sqrt{\det g} g^{ij} \partial_{x^j} f\right)$ and $\Delta_g f=\nabla_{x^i}\left(g^{ij} \nabla_{x^j} f \right)$ respectively.
\end{remark}

\subsection{Invariance of the elastic wave equation}

The following results concern local behavior of the EWE where the boundary is disregarded. These results address both the Euclidean and Riemannian setting and the global results can be derived from these.

By Remark~\ref{rem:Jonne}, in most cases any diffeomorphism that preserves the Finsler metric arising from elasticity in dimension two and three has to be conformal.
In the following proposition we write down explicitly how the Euclidean pushforward~$\phi_{\pfa} c$ by a diffeomorphism $\phi\colon M \to M$ is defined so that the principal behavior of the Euclidean EWE is preserved.

\begin{proposition}
\label{prop:CoordInvar}
The Euclidean and Riemannian elastic wave equations have the following invariance properties:
\begin{enumerate}
\item\label{part:rho-c-g}
$P_{(\rho,c,g)}=\phi^*P_{(\phi_* \rho,\phi_* c,\phi_* g)}\phi_*$ for any diffeomorphism~$\phi$.
\item\label{part:rho-c}
$P_{(\rho,c)}=\phi^*P_{(\phi_\pfcrho \rho,\phi_\pfcrho c)}\phi_*$ for any conformal map~$\phi$.
\item\label{part:c}
$\sigma(P_{(\rho,c)})=\sigma(P_{(1,\rho^{-1}c)})=\phi^*\sigma(P_{(1,\phi_\pfa (\rho^{-1}c))})\phi_*$ for any conformal map~$\phi$.
\end{enumerate}
\end{proposition}

The pullbacks~$\phi^*$ and pushforwards~$\phi_*$ surrounding the operators simply correspond to how the solution vector fields transform.

We only prove that in the Euclidean setting conformal diffeomorphisms preserve the desired structure of the stiffness tensor.
For the other direction --- that only conformal maps preserve such the symmetry --- we refer to the theorems and conjectures of section~\ref{sec:ThmConj}.

%\joonas{Upong reflection, I don't understand the following remark so I commented it out. See the comment there.}

%The division by~$\det D\phi$ in the definition of the pushforward~$\phi_{\pfcrho} c$ in~\eqref{cpushE} implies the following peculiar invariance in dimension $n=2$: 

%\begin{remark}
%\label{cor:Ptwo}
%\joonas{Now that we're in Euclidean geometry, $D\phi=\mu \, I_n$ is difficult to achieve. As we require $\phi$ to be conformal (otherwise there is little point in discussing the pushforward), I'm not sure if there are any nonlinear maps $\phi$ that satisfy the condition. The condition $D\phi=\mu \, I_n$ is not the definition of a conformal map.}
%Let $\phi\colon U \to V$ be so that $D\phi=\mu \, I_n$ then
%\begin{equation*}
%    (\phi_{\pfcrho} c)_{ijk\ell} = \mu^{2-n} c_{ijk\ell}
%\end{equation*}
%so it follows when $n=2$ that $P_{(\rho,c)}=P_{(\phi_{\pfcrho}\rho,\phi_{\pfcrho} c)}$.
%\end{remark}
%This is connected to the conformal invariance of the Laplace--Beltrami operator in dimension $n=2$.

Proposition~\ref{prop:CoordInvar} is largely based on the following lemma on scaling freedoms which may also be of independent interest:

\begin{lemma}
\label{lem:fullPDE}
The operator~$\Lop{}{c}{g}$ satisfies the following relationship for the scalar functions~$\mu$ and~$\lambda$ which satisfy $\mu\neq 0$ and $\lambda \neq 0$:
\begin{equation*}
    (\lambda \mu\rho)^{-1}(\Lop{}{\mu c}{ \lambda g})^i_{\hspace{1mm}m}= \rho^{-1}(\Lop{}{c}{g})^i_{\hspace{1mm}m} + Q^i_{\hspace{1mm}m}, 
\end{equation*}
where
\begin{equation*}
    Q^i_{\hspace{1mm}m}=\lambda^{-\frac{2+n}{2}} \mu^{-1} \rho^{-1} c^{ijk\ell}\, g_{\ell m}  \partial_{x^j} \lp \mu\lambda^{\frac{2+n}{2}}\rp \partial_{x^k}.
\end{equation*}
\end{lemma}

Table~\ref{tab:ConformalFreedoms} lists the consequences of Lemma~\ref{lem:fullPDE}, where we distinguish between principal and full behavior of the PDE and between the Euclidean and the Riemannian setting. For the full behavior we list the case when $\mu \lambda^{\frac{2+n}{2}}$ is a constant so that $Qu=0$.

\begin{table}[!ht]
    \centering
    \begin{tabular}{|p{3.1cm}|p{4.5cm}|p{6.1cm}|}
        \hline & Euclidean EWE & Riemannian EWE  \\ \hline 
        Principal behavior & One conformal freedom $\mu$: $(\rho,c)\sim(\mu \rho,\mu c)$  & Two conformal freedoms $\mu, \lambda$:\par $(\rho,c,g)\sim(\lambda \mu \rho,\mu c, \lambda g)$  \\ \hline
        Full behavior & No conformal freedom & One conformal freedom $\mu$:\par  $(\rho,c,g)\sim(\mu^{\frac{n}{2+n}} \rho,\mu c, \mu^{-\frac{2}{2+n}} g)$\\ \hline
    \end{tabular}
    \vspace{.3em}
    \caption{Consequences of Lemma~\ref{lem:fullPDE} for inverse problems.
    The relation~$\sim$ means that the two parameters give rise to the same boundary data on~$\partial\Omega$ or~$\partial M$.
    These are scaling freedoms, not diffeomorphism freedoms; the freedom to choose coordinates in the Euclidean setting is eliminated by the demand that boundary be fixed by the diffeomorphism.}
    \label{tab:ConformalFreedoms}
\end{table}

\section{Proofs}

\subsection{Proofs of auxiliary results}

\begin{proof}[Proof of Lemma~\ref{lem:fullPDE}]
A calculation gives
    \begin{align*}
        & (\lambda \mu\rho)^{-1}\lp \Lop{}{\mu c}{ \lambda g}u \rp^i \\
        &\hspace{6mm}= \lambda^{-1} \mu^{-1} \rho^{-1} \frac{1}{\sqrt{\det(\lambda g)}}  \partial_{x^j} \lp \sqrt{\det(\lambda g)} \, \mu \, c^{ijk\ell} \, \lambda \, g_{\ell m} \, \partial_{x^k} u^{m}\rp\\
        &\hspace{6mm}=\underbrace{\rho^{-1}\frac{1}{\sqrt{\det(g)}} \, \sum_{j=1}^n \partial_{x^j} \lp \sqrt{\det(g)} \, c^{ijk\ell} \, g_{\ell m} \, \partial_{x^k} u^{m} \rp}_{=\rho^{-1}(\Lop{}{c}{g}u)^i}\\
        &\hspace{12mm}+ \underbrace{\lambda^{-\frac{2+n}{2}} \mu^{-1} \rho^{-1} \sum_{j=1}^n c^{ijk\ell}\, g_{\ell m}  \partial_{x^j} \lp \mu\lambda^{\frac{2+n}{2}}\rp \, \partial_{x^k} u^{m}}_{=(Qu)^i}.
    \end{align*}
as claimed.
\end{proof}

\begin{proof}[Consequences of Lemma~\ref{lem:fullPDE} as listed in Table~\ref{tab:ConformalFreedoms}]The elliptic term $(\lambda \mu\rho)^{-1}\lp \Lop{}{\mu c}{ \lambda g}u \rp^i$ corresponds to the operator $P_{(\lambda \mu \rho, \mu c, \lambda g)}$. As $Q^i_{\hspace{1mm}m}$ is a lower order term it follows directly that the principal symbols of the  operators $P_{(\lambda \mu \rho, \mu c, \lambda g)}$ and $P_{(\rho,c,g)}$ are preserved:
    \begin{equation}
        \sigma(P_{(\lambda \mu \rho, \mu c, \lambda g)})=\sigma(P_{(\rho,c,g)})
    \end{equation}
    as highlighted in the upper right corner of Table~\ref{tab:ConformalFreedoms}. Conformal freedoms for the full Riemannian EWE can only be obtained in the case when $Qu=0$. Hence, when the term $\mu\lambda^{\frac{2+n}{2}}$ is constant. Choosing $\lambda = \mu^{-\frac{2}{2+n}}$ implies $\mu\lambda^{\frac{2+n}{2}}=1$. Hence, in this case 
    \begin{equation}\label{eq:derivFullRiem}
        \lp \mu^{\frac{n}{2+n}}\rho \rp^{-1} \Lop{}{\mu c}{ \mu^{-\frac{2}{2+n}} g}=\rho^{-1}(\Lop{}{c}{g}u)
    \end{equation}
    and thus
    \begin{equation}
        P_{\lp \mu^{\frac{n}{2+n}} \rho, \mu c, \mu^{-\frac{2}{2+n}} g\rp}=P_{(\rho,c,g)}
    \end{equation}
    as highlighted in the lower right corner of Table~\ref{tab:ConformalFreedoms}.\par 
    In the Euclidean case the metric is fixed; this eliminates the conformal freedom corresponding to the function~$\lambda$:
    \begin{equation*}
    (\mu\rho)^{-1}(\Lop{}{\mu c}{ g_E})_{i\ell}= \rho^{-1}(\Lop{}{c}{g_E})_{i\ell} + Q_{i\ell}, 
    \end{equation*}
    with
    \begin{equation*}
    Q_{i\ell}= \mu^{-1} \rho^{-1} c_{ijk\ell} \partial_{x^j} \lp \mu\rp \partial_{x^k}.
    \end{equation*}
    It follows that the principal behavior is preserved for the pairs $(\mu \rho, \mu c)$ and $(\rho, c)$ so that
    \begin{equation}
        \sigma(P_{(\mu \rho, \mu c)})=\sigma(P_{(\rho, c)}).
    \end{equation}
    As the full Euclidean EWE can only be preserved when~$\mu$ is constant it follows that there are no conformal freedoms in this case. These observations are highlighted in the upper and lower left corner respectively of Table~\ref{tab:ConformalFreedoms}.
\end{proof}

\begin{proof}[Proof of Proposition~\ref{prop:CoordInvar}]
\textbf{Part~(\ref{part:rho-c-g}):}
All structures used in this claim are tensor fields and covariant derivatives in the sense of Riemannian geometry, and therefore they are invariant under a change of coordinates by design.
Verification using equations~\eqref{eq:pf-c/rhocg}, \eqref{eq:pf-g/rhocg} and~\eqref{eq:pf-rho/rhocg} by hand is also straightforward.

\textbf{Part~(\ref{part:rho-c}):}
Note that a conformal map~$\phi$ with respect to the Euclidean geometry satisfies
\begin{equation}
\label{eq:cf-pf-gE2}
    \phi_*g_E
    =
    \det(D\phi)^{-2/n}g_E.
\end{equation}
Such a map is conformal whenever the Jacobian  $D\phi$ at each point is a scalar times a rotation matrix and the formula \eqref{eq:cf-pf-gE2} can for instance be derived from \eqref{eq:pf-g/rhocg} exploiting that fact.
Combining this with part~(\ref{part:rho-c-g}) gives
\begin{equation}
%\begin{split}
    \phi_* \lp\Lop{\rho^{-1}}{c}{g_E}\rp u 
    =
    \Lop{(\phi_*\rho)^{-1}}{\phi_*c}{\det(D\phi)^{-2/n}g_E}\phi_*u
%\end{split}
\end{equation}
for any~$u$.

We then apply Lemma~\ref{lem:fullPDE} with $\lambda=\det(D\phi)^{-2/n}$ and $\mu=\det(D\phi)^{1+2/n}$ (so that $\mu\lambda^{n/2+1}=1$) to get
\begin{equation}
%\begin{split}
    \Lop{(\phi_*\rho)^{-1}}{\phi_*c}{\det(D\phi)^{-2/n}g_E}
    %=
    %\Lop{(\phi_*\rho)^{-1}}{\phi_*c}{\lambda g_E}
    %=
    %\lambda\mu\Lop{(\lambda\mu\phi_*\rho)^{-1}}{\mu\mu^{-1}\phi_*c}{\lambda g_E}
    %=
    %\lambda\mu\Lop{(\phi_*\rho)^{-1}}{\mu^{-1}\phi_*c}{g_E}
    =
    \det(D\phi)\Lop{(\phi_*\rho)^{-1}}{\det(D\phi)^{-1-2/n}\phi_*c}{g_E}
%\end{split}
.
\end{equation}
The definitions in~\eqref{eq:pf-c/rhoc} and~\eqref{eq:pf-rho/rhoc} were set up so that $\phi_\pfcrho c=\det(D\phi)^{-1-2/n}\phi_*c$ and $\phi_\pfcrho \rho=\det(D\phi)^{-1}\phi_*\rho$.
Therefore
\begin{equation}
    \Lop{\rho^{-1}}{c}{g_E}u
    =
    \phi^* \lp \Lop{(\phi_\pfcrho \rho)^{-1}}{\phi_\pfcrho c}{g_E}\rp \phi_*u 
\end{equation}
from which the desired claim follows.

\textbf{Part~(\ref{part:c}):}
This is similar to part~(\ref{part:rho-c}) above, but we no longer need to satisfy $\mu\lambda^{n/2+1}=1$ as the subprincipal term of Lemma~\ref{lem:fullPDE} is irrelevant.
This corresponds to the principal symbol only depending on the density normalized stiffness tensor $\rho^{-1}c$, which is evident.
Our definition in~\eqref{eq:pf-c/c} satisfies
\begin{equation}
    \phi_\pfa c
    =
    (\phi_\pfcrho 1)^{-1}\phi_\pfcrho c
\end{equation}
with the constant density $1$, so the claim follows from essentially the same calculation as above.
\end{proof}

\subsection{Proof of the main result}

\begin{proof}[Proof of Theorem~\ref{thm:RiemConf}]
    As emphasized in part~(\ref{part:rho-c-g}) of Proposition~\ref{prop:CoordInvar}, the pushforwards of~$c$, $g$ and~$\rho$ defined in \eqref{eq:pf-c/rhocg}--\eqref{eq:pf-rho/rhocg} and the pushforward~$\phi_*u$ of the solution leave the Riemannian EWE invariant with respect to any diffeomorphism~$\phi$. Additionally we note that the DN map is invariant with respect to coordinate changes and with respect to the push forwards of $\rho, c$ and $g$ such a coordinate change can be expressed by:
    \begin{equation*}
        (\Lambda_{\left(\phi_* \rho, \phi_* c, \phi_* g \right)}f)^{\hat{i}}=(\phi_* \nu)_{\hat{j}} (\phi_* c)^{\hat{i}\hat{j}\hat{k}\hat{\ell}} (\phi_* g)_{\hat{\ell}\hat{m}} \partial_{x^{\hat{k}}} (\phi_* u)^{\hat{m}}
    \end{equation*}
    where,
    \begin{align*}
        (\phi_* \nu)_{\hat{j}}&=\lp \nu_j (D\phi^{-1})_{\hspace{2mm} \hat{j}}^j\rp \circ  \phi^{-1},\\
        (\phi_* u)^{\hat{m}}&=\lp u^m (D\phi)_{\hspace{2mm} m}^{\hat{m}}\rp \circ  \phi^{-1}.
    \end{align*}
    If $\phi$ is a diffeomorphism that fixes the boundary and thus is the identity in a neighborhood of $\partial M$ one can assume $D\phi(x)=I$ for $x$ in that neighborhood. This implies that $\phi_* \rho(x)=\rho(x), \phi_* c(x)=c(x), \phi_* g(x)=g(x), \phi_* \nu(x)=\nu(x)$ and $\phi_* u(x)=u(x)$ for $x$ in that neighborhood.  Therefore the triplet $(\phi_* \rho, \phi_* c, \phi_* g)$ gives rise to the same DN map as the triplet $(\rho, c, g)$:
    \begin{equation}\label{difDN}
        \Lambda_{\left(\phi_* \rho, \phi_* c, \phi_* g \right)}=\Lambda_{(\rho, c, g)}.
    \end{equation}
Additionally it was shown as a consequence of Lemma~\ref{lem:fullPDE} in \eqref{eq:derivFullRiem} that there is the following conformal freedom for the elastic wave operator $P_{(\rho,c,g)}$ for any non-vanishing function~$\mu$:
    \begin{equation}
    P_{\lp \mu^{\frac{n}{2+n}} \rho, \mu c, \mu^{-\frac{2}{2+n}} g\rp}=P_{(\rho,c,g)}.
    \end{equation}
    We note that 
    \begin{equation*}
        \Lambda_{\lp \mu^{\frac{n}{2+n}} \rho, \mu c, \mu^{-\frac{2}{2+n}} g\rp} = \mu^{\frac{n}{2+n}} \Lambda_{(\rho, c, g)}
        .
    \end{equation*}
    This implies for any function $\mu \neq 0$ that satisfies $\mu \vert_{\partial M}=1$:
    \begin{equation}\label{confDN}
        \Lambda_{\lp \mu^{\frac{n}{2+n}} \rho, \mu c, \mu^{-\frac{2}{2+n}} g\rp}=\Lambda_{(\rho, c, g)}.
    \end{equation}
    Combining~\eqref{difDN} and~\eqref{confDN} yields %the following gauge freedoms for the DN map:
    \begin{equation*}
        \Lambda_{\left(\mu^{\frac{n}{2+n}} \,\phi_* \rho, \mu \, \phi_* c, \mu^{-\frac{2}{2+n}} \, \phi_* g \right)}=\Lambda_{(\rho, c, g)}
    \end{equation*}
    as claimed.
\end{proof}

\appendix
\section{Relation to Euclidean, Riemannian, and Lorentzian Calderón problems}

Coordinate gauge freedom and conformal freedoms occur in some other inverse boundary value problems, and we discuss some examples to provide a point of comparison to our results.

The citations given in this appendix are only examples of the vast literature on these problems.
For more details and background, we refer the reader to~\cite{30years,Belishev_2007,Lassas18,KatchalovKurylevLassas01} and references therein.
%\joonas{Add a review or two. You said you had something in mind for the usual elliptic Calderon; it doesn't matter if it's a bit old. Possible Lorentzian citations to include:
%M. Belishev, Recent progress in the boundary control method, Inverse Problems, 23 (2007), R1–R67.
%M. Lassas, Inverse problems for linear and non-linear hyperbolic equations. Proc. Int. Cong. of Math. 2018, Rio de Janeiro, 3, 2018.
%A. Katchalov, Y. Kurylev, M. Lassas, Inverse boundary spectral problems, Chapman \& Hall/CRC Monogr. Surv. Pure Appl. Math., 2001.
%I rephrased the sentence so that we can give any relevant material (books, lecture notes, anything), not just reviews.
%}

%Non-standard gauge formula for the Euclidean $\gamma$ corresponds to a standard formula for $g$.

%Lorentzian Calderon naturally has conformal freedom.\\\\

\subsection{DN maps for Euclidean, Riemannian, and Lorentzian Calderón problems}
\label{app:dn}

The inverse problems we introduced in section~\ref{sec:dn-intro} are closely related to the Calderón problem concerned with the conductivity equation describing electrostatics:
\begin{equation*}
    \begin{cases}
        L_{\gamma}u(x):=\nabla \cdot (\gamma(x) \nabla u(x)) = 0 & \text{in }\Omega\\
        u = f & \text{on }\partial \Omega,
    \end{cases}
\end{equation*}
where~$\gamma$ denotes an isotropic or anisotropic conductivity. The DN map corresponding to~$L_{\gamma}$ is defined by
\begin{equation}\label{DNmapCe}
    \Lambda_{\gamma} f= \nu_i \, \gamma_{ij} \, \partial_{x_j} u \Bigg\vert_{\partial \Omega}.
\end{equation}
The Calderón problem is then concerned with recovering~$\gamma$ from~$\Lambda_{\gamma}$.
For an anisotropic conductivity~$\gamma$ it turns out that this inverse problem is closely related to a corresponding geometric inverse problem. Consider the Dirichlet problem associated to the Laplace--Beltrami operator on $(\Omega,g)$
\begin{equation*}
    \begin{cases}
        \Delta_{g} u = 0 & \text{in }\Omega\\
        u = f & \text{on }\partial \Omega,
    \end{cases}
\end{equation*}
where (cf. Remark~\ref{rmk:cov-div} for an alternative formula)
\begin{equation}
    \Delta_g u = \frac{1}{\sqrt{\det(g)}}\partial_{x^{i}} \lp \sqrt{\det(g)} g^{ij}  \partial_{x^j} u\rp
\end{equation}
and with corresponding DN map defined by
\begin{equation*}
\label{DNmapCr}
    \Lambda_{g} f= \nu_i \, g^{ij} \, \partial_{x^j} u \Bigg\vert_{\partial \Omega}.
\end{equation*}
Then the geometric Calderón problem asks to recover~$g$ from~$\Lambda_g$. One can also consider the following geometric version of the conductivity equation on a Riemannian manifold $(M,g)$:
\begin{equation*}
    \begin{cases}
        \text{div}_g(\beta \nabla_g u) = 0 & \text{in }M\\
        u = f & \text{on }\partial M,
    \end{cases}
\end{equation*}
where~$\beta$ is an isotropic conductivity. In this setting the DN map is defined as 
\begin{equation}\label{DNmapRiem}
    \Lambda_{(\beta,g)} f=\nu_i \, \beta \, g^{ij} \, \partial_{x^j} u \Bigg\vert_{\partial M}.
\end{equation}
This gives rise to the Calderón problem on a Riemannian manifold that asks to recover~$\beta$ and~$g$ from~$\Lambda_{(\beta,g)}$.

There is also a similarity to the Lorentzian Calderón problem in that there is a conformal gauge freedom in all dimensions.
An example is the non-linear wave equation
\begin{equation*}
    \begin{cases}
        \square_g u(x) + a(x) u(x)^4 = 0, & \text{on }M,\\
        u(x)=f(x), & \text{on }\partial M,\\
        u(t,x')=0, & t<0,
    \end{cases}
\end{equation*}
where $\square_g$ is the d'Alembert operator of~$g$ (the Laplace--Beltrami operator on a Lorentzian manifold) and $M=\R \times N$.
The corresponding DN map is defined as
\begin{equation}\label{DNmapLor}
    \Lambda_{(g,a)}f=\nu^j \partial_{x^j} u \vert_{(0,T)\times \partial N}
\end{equation}
and the Lorentzian Calderón problem asks to recover the metric~$g$ and the function~$a$ from~$\Lambda_{(g,a)}$.
%\hjordis{Maybe it gets a bit too much with also the Calderón problem on the manifold and the Lorentzian Calderon problem here?}
%\joonas{I think it's fine. Now that it's in an appendix, it's skippable. But it might be interesting to some.}

\subsection{Gauge freedom in the Euclidean and Riemannian Calderón problems}

It was observed by L. Tartar that the map~$\Lambda_{\gamma}$ defined in~\eqref{DNmapCe} does not determine~$\gamma$ uniquely (see~\cite{KohnVog84} for an account). This is due to the fact that any~$C^{\infty}$ diffeomorphism~$\phi\colon \overline{\Omega} \to \overline{\Omega}$ with $\phi \vert_{\partial \Omega}=\operatorname{id}$ and with the Euclidean pullback %\joonas{This is a pullback, not a pushforward. Perhaps denote it $\phi^{\pfcrho}$ instead? Or change the formula from ${\det( D \phi)} D \phi^{-1} \, \gamma \, D \phi^{-T}$ to $\frac{1}{\det( D \phi)} D \phi \, \gamma \, D \phi^t$, which means just replacing $\phi$ with $\phi^{-1}$. The pushforward is the same as the pullback over the inverse.} 
of the conductivity
\begin{equation}
\label{eq:pf-calderon}
    \phi^{\pfcrho} \gamma
    =
    \frac{1}{\det( D \phi)} D \phi \, \gamma \, D \phi^t, 
\end{equation}
produces the same DN map as~$\gamma$:
\begin{equation}
    \Lambda_{\phi^{\pfcrho} \gamma} = \Lambda_{\gamma}.
\end{equation}
%Notice that the pushforwards\joonas{Possibly rephrase here, saying that the pushforward of one eq. and the pullback of the other eq. are not the standard ones. Cf. footnote above.} defined in~\eqref{eq:pf-calderon} and~\eqref{eq:pf-c/rhoc} are not the usual pushforwards of tensor fields in differential geometry.
Notice that the pullback defined in~\eqref{eq:pf-calderon} and the pushforward defined in~\eqref{eq:pf-c/rhoc} are not the usual pullbacks and pushforwards of tensor fields in differential geometry.

Similarly to the Euclidean case~$\Lambda_g$ as defined in~\eqref{DNmapCr} does not determine~$g$ uniquely due to the coordinate gauge freedom
\begin{equation}
    \Lambda_{\phi_* g} = \Lambda_{g},
\end{equation}
where~$\phi_*$ denotes the classical pushforward of~$g$ by~$\phi$.
It was shown in~\cite{LasUhl01} that this is the only gauge freedom for real analytic metrics in dimension $n\geq 3$.

In dimension $n=2$ the Laplace--Beltrami operator is conformally invariant which gives rise to an additional gauge for the inverse problem:
\begin{equation*}
    \Lambda_{\alpha (\phi_{*} g)} = \Lambda_g.
\end{equation*} 
This is proven by~\cite{LasUhl01} to be the only obstruction for unique identifiability of~$g$.

The inverse problems for the DN maps of~\eqref{DNmapCe} and~\eqref{DNmapCr} are equivalent when $n\geq3$.
Namely, if we set
\begin{equation*}
\begin{split}
    g &= (\det \gamma)^{\frac{1}{n-2}} \gamma^{-1}
    \quad \text{or equivalently} \\
    \gamma &= (\det \, g)^{\frac{1}{2}} g^{-1},
\end{split}
\end{equation*}
then (see~\cite{LeeUhlmann89})
\begin{equation*}
    \Lambda_{g}=\Lambda_{\gamma},
\end{equation*}
if we express $\Lambda_g$ with respect to the Euclidean surface measure:
\begin{equation*}
    \Lambda_{g} f= \sqrt{\det g}\,\nu_i \, g^{ij} \, \partial_{x^j} u \Bigg\vert_{\partial \Omega}.
\end{equation*}

\subsection{Gauge freedom in the Calderón problem on manifolds}

It was shown in~\cite{SunUhlmann03} that the conformal freedom of the Laplace--Beltrami operator in two dimensions yields the following conformal freedom for the DN operator defined in~\eqref{DNmapRiem} in two dimensions for $\phi\colon M \to M$ with $\phi \vert_{\partial M}=\operatorname{id}$:
\begin{equation*}
    \Lambda_{(\alpha (\phi_{*} g), \phi_* \beta)} = \Lambda_{(g,\beta)},
\end{equation*}
for any positive scalar function~$\alpha$. 

\subsection{Gauge freedom in the Lorentzian Calderón problem}

Similarly to the previous Calderón problems there is a gauge freedom by a diffeomorphism for the DN map defined in~\eqref{DNmapLor}.
It was shown in~\cite{HintzUhlmannZhai22} that for $\phi\colon M \to M$ with $\phi \vert_{\partial M}=\operatorname{id}$ we have
\begin{equation*}
    \Lambda_{(\alpha (\phi^{*} g), \phi^* a)} = \Lambda_{(g,a)}
\end{equation*}
and that there is the conformal freedom
\begin{equation*}
    \Lambda_{(e^{-2\beta}g,e^{-\beta}a)}=\Lambda_{(g,a)}
\end{equation*}
for a smooth function~$\beta$ on~$M$ such that
\begin{equation*}
    \beta \vert_{\partial M}=0, \quad \partial_{\nu} \beta \vert_{\partial M} = 0, \quad \square_g e^{-\beta}=0.
\end{equation*}
%\hjordis{Are there more references you were thinking of with respect to the Lorentzian Calderón problem?}
%\joonas{See the footnote at the beginning of the appendix.}

%\subsection{New stuff}

%\joonas{!}

%\clearpage
%\newpage
\printbibliography

@article {HILSmay21,
    AUTHOR = {de Hoop, Maarten V. and Ilmavirta, Joonas and Lassas, Matti
              and Saksala, Teemu},
     TITLE = {A foliated and reversible {F}insler manifold is determined by
              its broken scattering relation},
   JOURNAL = {Pure Appl. Anal.},
    VOLUME = {3},
      YEAR = {2021},
    NUMBER = {4},
     PAGES = {789--811},
       DOI = {10.2140/paa.2021.3.789},
}

@phdthesis{stolk00,
    author =   {Stolk,Christiaan Cornelis},
    title = {On the modeling and inversion of seismic data},
    school = {Utrecht University},
    year = {2000}
}

@misc{IlmavirtaCocan,
    author = {Antonio Cocan and Joonas Ilmavirta},
    note = {In preparation.}
}

@article{Stefanov2017,
  title = {Local recovery of the compressional and shear speeds from the hyperbolic DN map},
  volume = {34},
  DOI = {10.1088/1361-6420/aa9833},
  number = {1},
  journal = {Inverse Problems},
  publisher = {IOP Publishing},
  author = {Stefanov,  Plamen and Uhlmann,  Gunther and Vasy,  Andras},
  year = {2017},
  month = dec,
  pages = {014003}
}

@article{deHoop2019,
  title = {Recovery of Material Parameters in Transversely Isotropic Media},
  volume = {235},
  DOI = {10.1007/s00205-019-01421-5},
  number = {1},
  journal = {Archive for Rational Mechanics and Analysis},
  publisher = {Springer Science and Business Media LLC},
  author = {de Hoop,  Maarten V. and Uhlmann,  Gunther and Vasy,  András},
  year = {2019},
  month = jul,
  pages = {141–165}
}

@article{BeLa02,
  title = {The dynamical Lame system: regularity of solutions,  boundary controllability and boundary data continuation},
  volume = {8},
  ISSN = {1262-3377},
  url = {http://dx.doi.org/10.1051/cocv:2002058},
  DOI = {10.1051/cocv:2002058},
  journal = {ESAIM: Control,  Optimisation and Calculus of Variations},
  publisher = {EDP Sciences},
  author = {Belishev,  M. I. and Lasiecka,  I.},
  year = {2002},
  pages = {143–167}
}

@article{Rachele001,
author = { Lizabeth   V Rachele },
title = {Boundary determination for an inverse problem in elastodynamics},
journal = {Communications in Partial Differential Equations},
volume = {25},
number = {11-12},
pages = {1951-1996},
year  = {2000},
publisher = {Taylor & Francis},
doi = {10.1080/03605300008821575},
}

@article{RACHELE002,
title = {An Inverse Problem in Elastodynamics: Uniqueness of the Wave Speeds in the Interior},
journal = {Journal of Differential Equations},
volume = {162},
number = {2},
pages = {300-325},
year = {2000},
issn = {0022-0396},
doi = {https://doi.org/10.1006/jdeq.1999.3657},
author = {Lizabeth V. Rachele},
}

@article{HintzUhlmannZhai22,
author = {Peter Hintz and Gunther Uhlmann and Jian Zhai},
title = {The Dirichlet-to-Neumann map for a semilinear wave equation on Lorentzian manifolds},
journal = {Communications in Partial Differential Equations},
volume = {47},
number = {12},
pages = {2363-2400},
year  = {2022},
publisher = {Taylor & Francis},
doi = {10.1080/03605302.2022.2122837},
}

@article{HoopNakamuraZhai17,
author = {de Hoop, Maarten V. and Nakamura, Gen and Zhai, Jian},
title = {Reconstruction of Lamé Moduli and Density at the Boundary Enabling Directional Elastic Wavefield Decomposition},
journal = {SIAM Journal on Applied Mathematics},
volume = {77},
number = {2},
pages = {520-536},
year = {2017},
doi = {10.1137/16M1073406},
}

@article{HoopNakamuraZhai19,
author = {de Hoop, Maarten V. and Nakamura, Gen and Zhai, Jian},
title = {Unique Recovery of Piecewise Analytic Density and Stiffness Tensor from the Elastic-Wave Dirichlet-To-Neumann Map},
journal = {SIAM Journal on Applied Mathematics},
volume = {79},
number = {6},
pages = {2359-2384},
year = {2019},
doi = {10.1137/18M1232802},
}

@article{Rachele03,
 ISSN = {00029947},
 URL = {http://www.jstor.org/stable/1194768},
 author = {Lizabeth V. Rachele},
 journal = {Transactions of the American Mathematical Society},
 number = {12},
 pages = {4781--4806},
 publisher = {American Mathematical Society},
 title = {Uniqueness of the Density in an Inverse Problem for Isotropic Elastodynamics},
 urldate = {2023-09-15},
 volume = {355},
 year = {2003}
}

@incollection {KohnVog84,
    AUTHOR = {Kohn, Robert V. and Vogelius, Michael},
     TITLE = {Identification of an unknown conductivity by means of
              measurements at the boundary},
    SERIES = {SIAM-AMS Proc.},
    VOLUME = {14},
     PAGES = {113--123},
 PUBLISHER = {Amer. Math. Soc., Providence, RI},
      YEAR = {1984},
       DOI = {10.1002/cpa.3160370302},
}

@article {LeeUhlmann89,
    AUTHOR = {Lee, John M. and Uhlmann, Gunther},
     TITLE = {Determining anisotropic real-analytic conductivities by
              boundary measurements},
   JOURNAL = {Comm. Pure Appl. Math.},
    VOLUME = {42},
      YEAR = {1989},
    NUMBER = {8},
     PAGES = {1097--1112},
       DOI = {10.1002/cpa.3160420804},
}

@article {SunUhlmann03,
    AUTHOR = {Sun, Ziqi and Uhlmann, Gunther},
     TITLE = {Anisotropic inverse problems in two dimensions},
   JOURNAL = {Inverse Problems},
    VOLUME = {19},
      YEAR = {2003},
    NUMBER = {5},
     PAGES = {1001--1010},
       DOI = {10.1088/0266-5611/19/5/301},
}

@article{30years,
     author = {Uhlmann, Gunther},
     title = {30 {Years} of {Calder\'on{\textquoteright}s} {Problem}},
     journal = {S\'eminaire Laurent Schwartz {\textemdash} EDP et applications},
     note = {talk:13},
     publisher = {Institut des hautes \'etudes scientifiques & Centre de math\'ematiques Laurent Schwartz, \'Ecole polytechnique},
     year = {2012-2013},
     doi = {10.5802/slsedp.40},
     language = {en},
     %url = {http://www.numdam.org/articles/10.5802/slsedp.40/}
}

@article{Belishev_2007,
doi = {10.1088/0266-5611/23/5/R01},
%url = {https://dx.doi.org/10.1088/0266-5611/23/5/R01},
year = {2007},
month = {sep},
publisher = {},
volume = {23},
number = {5},
pages = {R1},
author = {M I Belishev},
title = {Recent progress in the boundary control method},
journal = {Inverse Problems}
}

@inbook{Lassas18,
author = { Matti Lassas},
title = {Inverse problems for linear and non-linear hyperbolic equations},
booktitle = {Proceedings of the International Congress of Mathematicians (ICM 2018)},
chapter = {},
year = {2018},
pages = {3751-3771},
doi = {10.1142/9789813272880_0199},
%URL = {https://www.worldscientific.com/doi/abs/10.1142/9789813272880_0199},
%eprint = {https://www.worldscientific.com/doi/pdf/10.1142/9789813272880_0199},
}

@book{KatchalovKurylevLassas01,
title = "Inverse boundary spectral problems",
author = "Alexander Katchalov and Yaroslav Kurylev and Matti Lassas",
year = "2001",
doi = "10.1201/9781420036220",
language = "English",
isbn = "1-58488-005-8",
series = "Chapman amp; Hall/CRC Monographs and Surveys in Pure and Applied Mathematics",
publisher = "Chapman \& Hall / CRC",
address = "International",
}

@misc{HILSaug21,
  doi = {10.48550/arXiv.1901.03902},
  author = {de Hoop, Maarten V. and Ilmavirta, Joonas and Lassas, Matti
              and Saksala, Teemu},
  title = {Determination of a compact Finsler manifold from its boundary distance map and an inverse problem in elasticity},
  year = {2019},
  note = {Preprint, to appear in \textit{Comm. Anal. Geom.} arXiv},
}

@article {LasUhl01,
    AUTHOR = {Lassas, Matti and Uhlmann, Gunther},
     TITLE = {On determining a {R}iemannian manifold from the
              {D}irichlet-to-{N}eumann map},
   JOURNAL = {Ann. Sci. \'{E}cole Norm. Sup. (4)},
    VOLUME = {34},
      YEAR = {2001},
    NUMBER = {5},
     PAGES = {771--787},
       DOI = {10.1016/S0012-9593(01)01076-X},
}

@misc{dehoop2023reconstruction,
      title={Reconstruction of generic anisotropic stiffness tensors from partial data around one polarization}, 
      author={Maarten V. de Hoop and Joonas Ilmavirta and Matti Lassas and Anthony Várilly-Alvarado},
      year={2023},
      eprint={2307.03312},
      archivePrefix={arXiv},
      primaryClass={math.DG}
}

%\newpage

%\input{Appendix.tex}

\end{document}